\documentclass{article}

\usepackage{amssymb,latexsym}
\usepackage{amsmath}
\usepackage{graphics}
\usepackage{graphicx}
\newtheorem{theorem}{Theorem}

\newtheorem{definition}[theorem]{Definition}

\newtheorem{lemma}{Lemma}

\newenvironment{proof}[1][Proof]{\textbf{#1.} }{\ \rule{0.5em}{0.5em}}
\numberwithin{equation}{section} \numberwithin{theorem}{section}
\numberwithin{corollary}{section} \numberwithin{lemma}{section}
\numberwithin{remark}{section} \numberwithin{notation}{section}

\begin{document}

\title{Perfect Morse Function on $SO(n)$}

\author{Mehmet Solgun \\
Bilecik Seyh Edebali University, Faculty of Sciences and Arts\\
 Department of Mathematics, Bilecik, TURKEY\\
 mehmet.solgun@bilecik.edu.tr}

\maketitle

\begin{abstract}
In this work, we define a Morse function on $SO(n)$ and show that this function is indeed a perfect Morse function.
\end{abstract}


\section{Introduction}

The main point of Morse Theory, which was introduced in \cite{Morse1928}, is investigating the relation between shape of a smooth manifold $M$ and critical points of a specific real-valued function $f:M\rightarrow \mathbb{R}$, that is called \textit{Morse function.} \cite{Milnor1963} and \cite{Matsumoto2002} are two of main sources about this subject, so mostly we will use their beautiful tools for defining a Morse function on $SO(n)$. Also, we will refer \cite{Hatcher2002} to use homological properties   and to determine the Poincar\'e polynomial of $SO(n)$. Perfect Morse functions are widely studied in \cite{Nicolaescu2011}, that is one of our inspiration to show that the function, we defined, is also perfect.

\section{Preliminaries}
In this section, we give some definitions and theorems which will be used in this paper.

\begin{definition} Let $M$ be an $n$-dimensional  smooth manifold and $f:M\rightarrow\mathbb{R}$ be a smooth function. A point $p_0 \in M$ is said to be \textit{a critical point of $M$} if we have
\begin{equation}
\frac{\partial f}{\partial x_1}=0, \: \frac{\partial f}{\partial x_2}=0,\:...,\:  \frac{\partial f}{\partial x_n}=0
\end{equation}
with respect to a coordinate system $\lbrace x_1,x_2,...,x_n\rbrace$ around $p_0$.
\end{definition}

A point $c\in \mathbb{R}$ is said to be  \textit{a critical value
 of} $f:M\rightarrow\mathbb{R}$, if $f(p_0)=c$ for a critical point $p_0$ of $f$.
\begin{definition}
Let $p_0$ be a critical point of the function $f:M\rightarrow \mathbb{R}$. The \textit{Hessian of $f$ at he point $p_0$} is the $n\times n$ matrix
\begin{equation}
H_f(p_0)=\begin{bmatrix}
\frac{ \partial^2f}{\partial{x_1}^2}(p_0) & \cdots & \frac{\partial^2 f}{\partial x_1  \partial x_n}(p_0) \\
\vdots                                    & \ddots & \vdots                                               \\
\frac{\partial^2 f}{\partial x_n  \partial x_1}(p_0) & \cdots &  \frac{ \partial^2f}{\partial{x_n}^2} \: (p_0)
\end{bmatrix}
\end{equation}
\end{definition}
Since $\frac{\partial ^2 f}{\partial x_i \partial x_j}(p_0)=\frac{\partial ^2 f}{\partial x_j \partial x_i}(p_0)$, the Hessian of $f$ is a symmetric matrix.

Let $p_0$ be a critical point of $f$ and $c_0 \in \mathbb{R}$ such that $f(p_0)=c_0$. Then, $c_0$ is said to be \textit{a critical value of $f$}. If $p_0$ is a regular point of $f$, then $c_0$ is said to be a \textit{regular value of $f$}.

If $a$ is a regular value of $f$, it can be shown that the set $f^{-1}(a)=\lbrace p\in M| f(p)=a \rbrace$ is an $n-1$ dimensional manifold \cite{Gauld1982}.

\begin{definition}
A critical point of a function $f:M\rightarrow\mathbb{R}$ is called "non-degenerate point of $f$" if $det H_f(p_0)\neq 0$. Otherwise, it is called "degenerate critical point".
\end{definition}

\begin{lemma}
Let $p_0$ be a critical point of a smooth function $f:M \rightarrow \mathbb{R}$, $(U,\varphi=(x_1,...,x_n))$, $(V,\psi=(X_1,...,X_n))$ be two charts of $p_0$, and  $ H_f(p_0),\mathcal{H}_f(p_0)$ be the Hessians of $f$ at $p_0$, using the charts $(U,\varphi), (V,\psi)$ respectively. Then the following holds:
\begin{equation}
 \mathcal{H}_f(p_0)=J(p_0)^tH_f(p_0)J(p_0)
\end{equation}
where $J(p_0)$ is the Jacobian matrix for the given coordinate transformation, defined by
\begin{equation}
J(p_0)= \begin{bmatrix}
\frac{\partial x_1}{\partial X_1}(p_0) & \cdots & \frac{\partial x_1}{\partial X_n}(p_0) \\
\vdots                          & \ddots & \vdots \\
\frac{\partial x_n}{\partial X_1}(p_0) & \cdots & \frac{\partial x_n}{\partial X_n}(p_0)
\end{bmatrix}
\end{equation}
and the matrix $J(p_0)^t$ is the transpose of $J(p_0)$.
\end{lemma}
For a critical point $p_0$, non-degeneracy does not depend on the choice of charts around $p_0$. The same argument is also true for degenerate critical points. In fact we have
\begin{equation*}
 \mathcal{H}_f(p_0)=J(p_0)^tH_f(p_0)J(p_0)
\end{equation*}
by the previous lemma,and hence
\begin{equation}
det \mathcal{H}_f(p_0)=det J(p_0)^t det H_f(p_0) det J(p_0)
\end{equation}
by using determinant function on both sides. On the other hand, the determinant of the Jacobian matrix is non-zero. So the statement "$det \mathcal{H}_f(p_0)\neq0$" and "$ det H_f(p_0)\neq0$" are equivalent. In other words,
\begin{equation*}
det \mathcal{H}_f(p_0)\neq0 \Leftrightarrow det H_f(p_0)\neq0.
\end{equation*}
Now a function $f:M\rightarrow\mathbb{R}$ is called a \textit{Morse function} if any critical point
of f is non-degenerate. From now on, we only consider a Morse function $f$.

Now, we introduce Morse lemma on manifolds.

\begin{theorem}{(The Morse Lemma)}
Let $M$ be an $n$-dimensional smooth manifold and $p_0$ be a non-degenerate critical point of a Morse function $f:M\rightarrow\mathbb{R}$. Then, there exists a local coordinate system $(X_1,X_2,...,X_n)$ around $p_0$ such that the coordinate representation of $f$ has the following form:
\begin{equation} \label{f}
f=-X_1^2-X_2^2-...-X_\lambda^2+X_{\lambda+1}^2+...+X_n^2+c
\end{equation}
where $c=f(p_0)$ and $p_0$ corresponds to the origin $(0,0,...,0)$.
\end{theorem}

One may refer to see \cite{Milnor1963} for the proof.
 \\
The number $\lambda$ of minus signs in the equation \eqref{f} is the number of negative diagonal entries of the matrix $H_f(p_0)$ after diagonalization. By Sylvester's law, $\lambda$ does not depend on how $H_f(p_0)$ is diagonalized. So, $\lambda$ is determined by $f$ and $p_0$.
The number $\lambda$ is called "the index of the non-degenerate critical point $p_0$". Obviously, $\lambda$ is an integer between $0$ and $n$.

Note that,

\begin{enumerate}
\item A non-degenerate critical point is isolated.
\item A Morse function on a compact manifold has only finitely many critical points \cite{Milnor1963}.
\end{enumerate}

\section{A Morse function on $SO(n)$}
In this section, we will  define a Morse function on $SO(n)$.

The set of all $n\times n$ orthogonal matrices,  $O(n)= \lbrace A=(a_{ij}) \in M_n (\mathbb{R}): AA^t=I_n \rbrace $ is a group with matrix multiplication.
 From the definition of $O(n)$,
  \begin{center}
 $detA= \pm 1$, for any $A \in O(n).$
 \end{center}

 An orthogonal matrix with determinant $1$ is called \textit{rotation matrix} and  the set of this kind of matrices is  also a group, called \textit{special orthogonal group} and denoted by $SO(n)$. On the other hand, let $S_n(\mathbb{R})$ denote the set of symmetric $ n\times n $ matrices. Since each symmetric matrix is uniquely determined by its entries on and above the main diagonal, that is a linear subspace of $M_n(\mathbb{R})$ of dimension $n(n+1)/2$.

 Now we define a function $ \varphi : GL_n(\mathbb{R}) \longrightarrow S_n(\mathbb{R})$ by $$ \varphi(A):=A^tA .$$

Then, the identity matrix $I_n$ is a regular value of $\varphi$ \cite{Lee2012}.

Let $C \in S_n(\mathbb{R})$ with entries $c_i$, with $0\leq c_1<c_2<...<c_n$ fixed real numbers and  $f_C:SO(n)\rightarrow\mathbb{R}$ be given by,
\begin{equation}
f_C(A):=<C,A>=c_1x_{11}+c_2x_{22}+...+c_nx_{nn},
\end{equation}
where $A=(x_{ij}) \in SO(n)$.

 Obviously, $f_C$ is a smooth function. Now, we will determine its critical points.

\begin{lemma} The critical points of the function $f_C$ defined above are:
\begin{equation} \label{f_C}
\begin{bmatrix}
\pm 1 &    0   & \cdots &0 \\
0     & \pm 1  & \ddots & \vdots \\
\vdots & \ddots & \ddots & 0 \\
0      & \cdots  & 0    & \pm 1
\end{bmatrix}
\end{equation}
\end{lemma}

 \begin{proof}

Let $A$ be a critical point of $f_C$. Then the derivative of $f_C$ at $A$ must be zero. Consider the matrix given by a rotation of first and second coordinate $B_{12}(\theta)$ defined by
\begin{equation*}
B_{12}(\theta)=\begin{bmatrix}
cos\theta & -sin\theta &0 &\cdots & 0 \\
sin\theta & cos\theta  &0 &\cdots &0  \\
0        & 0           &1  &    &  \vdots\\
\vdots & \vdots & & \ddots & \vdots \\
0 & 0& \cdots &0& 1
\end{bmatrix}.
\end{equation*}
Then, $AB_{12}(\theta)\in SO(n)$ and the matrix $B_{12}(\theta)$ forms a curve on $SO(n)$. Moreover, $B_{12}(\theta)=A$ for $\theta=0$.

By the definition of $f_C$, and after computing the matrix product, we have
\begin{equation}
f_C(AB_{12}(\theta))=c_1(x_{11}cos\theta+x_{12}sin\theta)+c_2(-x_{21}sin\theta+x_{22}cos\theta)+c_3x_{33}+\cdots+c_nx_{nn}
\end{equation}

By differentiating $f$ in the direction of the velocity vector $\frac{d}{d\theta}AB_{12}(\theta)|_{\theta=0}$ of the curve $AB_{12}(\theta)$ at $A$, we have
\begin{equation}
\frac{d}{d\theta}f_C(AB_{12}(\theta))|_{\theta=0}= c_1x_{12}-c_2x_{21}
\end{equation}
 and
 \begin{equation}
 \frac{d}{d\theta}f_C(B_{12}(\theta)A)|_{\theta=0}= -c_1x_{21}+c_2x_{12}.
 \end{equation}

 However, by the assumption that $A$ is a critical point of $f_C$, we require these derivatives to be zero. i.e.

 \begin{equation*}
 c_1x_{12}-c_2x_{12}=0
 \end{equation*}
 \begin{equation*}
 -c_1x_{21}+c_2x_{12}=0
 \end{equation*}
 Solving this system for $ x_{12}, x_{21}$ gives $x_{12}= x_{21}= 0$. We can carry out the similar calculation for $B_{ij}(\theta)$ with $i<j$, where $B_{ij}(\theta)$ is with the entries: $(i,i)=cos\theta, (i,j)=-sin\theta, (j,i)=sin\theta$ and $(j,j)=cos\theta$. Thus, for the matrix $A$, $x_{ij}=0$ whenever $i\neq j$. So, that is, a critical point of $f_C$ is a diagonal matrix. On the other hand $A\in SO(n)$, so we have $AA^t=I_n$. So each entry on the main diagonal of $A$ must be $\pm1$.

 Conversely, let $A$ be a matrix in the form \eqref{f_C}. In order to check that $A$ is  a critical point, we need to compute the derivative of $f_C$. If we could find $n(n-1)/2$ curves $C_i$ going through $A$ with  velocity vector at $A$ and linearly independent from each other. Since the velocity vector of $C_i$ at $A$ plays a role of a local coordinate of $A$, we only need to check that the derivative of $f_C(C_i)$ vanishes to see that $Df(A)=0$. Now, the claim is the curves $C_i$'s are in fact $AB_{ij}\theta$' s defined above. Let $\epsilon_i=A_{ii}$ where $A_{ii}$ is the $i$-th diagonal entry of $A (\epsilon_i = \pm 1)$. Then, the derivative of the matrix $AB_{ij}(\theta)$ at $A$ is ( we did for the case $B_{12}$, but it is same for other indices with $i<j$),
 \begin{equation*}
 \frac{d}{d\theta}AB_{12}(\theta)|_{\theta=0}=\begin{bmatrix}
 0&-\epsilon_1 &0&\cdots &0 \\
 \epsilon_2 &0 &0 &\cdots&0 \\
 0&0 &\ddots& & 0 \\
 \vdots & \vdots & &  & \vdots \\
 0&0 &\cdots & & 0
 \end{bmatrix}
 \end{equation*}
 This matrix is regarded as a vector in $\mathbb{R}^{n^2}$. By considering all $1\leq i \leq j \leq n$, these matrices (vectors) form a basis for the tangent space $T_ASO(n)$.

 So, for a given matrix $A$ in the form (3.2), it is easy to compute that, the derivative of $f_C$ at $A$ is zero. This means nothing but $A$ is a critical point of $f_C$.
  \end{proof}

After now, we know the coordinate system of $SO(n)$ and the critical points of the the given function $f_C$. It is straightforward to compute the Hessian of $f_C$ at $A$. Suppose that $A$ is a critical matrix with diagonal entries $A_{ii}=\epsilon_i=\pm1$. Then, we want to compute
\begin{equation*}
\frac{\partial^2}{\partial\theta\partial\varphi}f_C(AB_{\alpha\beta}(\theta)B_{\gamma\delta}(\varphi))|_{\theta=0, \varphi=0}.
\end{equation*}
Notice that  is linear $AB_{\alpha\beta}(\theta)B_{\gamma\delta}(\varphi)$ is linear in $\theta$ and in $\varphi$, and $f_C$ is a linear function. Thus, we can bring the derivative inside $f_C$. So,

\begin{align*}
 \frac{\partial^2}{\partial\theta\partial\varphi}f_C(AB_{\alpha\beta}(\theta)B_{\gamma\delta}(\varphi))|_{\theta=0, \varphi=0} &= f_C(A\frac{d}{d\theta}B_{\alpha\beta}(\theta)|_{\theta=0}\frac{d}{d\varphi}B_{\gamma\delta}(\varphi)|_{\varphi=0}\\
 &= \left\{ \begin{array}{ll}
 -c_\alpha\epsilon_\alpha-c_\beta\epsilon_\beta & \textrm {if} \: \alpha=\gamma, \: \beta=\delta \\
 0 & \textrm{otherwise}
 \end{array}\right.
\end{align*}

 This calculation becomes easier if we consider the matrix multiplication $c_{ij}=\sum _ka_{ik}b_{kj}$. The calculation above shows that the Hessian matrix is diagonal. Since $c_\alpha \neq c_\beta$ for $\alpha \neq \beta$, the entries on the diagonal is non-zero. Therefore, $A$ is a non-degenerate critical point of $f_C$, meaning that $f_C$ is a Morse function on $SO(n)$.

Assume that the subscripts $i$ of the diagonal entries $\epsilon_i$ of $A$ , $1\leq i\leq n$, with $\epsilon_i=1$ are
\begin{equation*}
i_1,i_2,\cdots,i_m
\end{equation*}
in ascending order. Then the index of the critical point $A$ ( the number of minus signs on the diagonal of Hessian) is
\begin{equation*}
 (i_1-1)+(i_2-1)+\cdots+(i_m-1).
\end{equation*}
And the index is 0 if all $\epsilon_i$' s are -1. Also, the critical value at the critical point is
\begin{equation*}
2(c_{i1}+c_{i2}+\cdots+c_{im})-\sum _{i=0}^nc_i.
\end{equation*}
 Considering that $detA=1$, there are $2^{n-1}$ critical points \cite{Matsumoto2002}
 
\section{Perfect Morse Functions}
First, we will give the basic notions.
\begin{definition}
The Poincar\'e polynomial of the $n$- dimensional manifold $M$ is defined to be
\begin{equation}
P_M(t)=\sum_{k=0}^nb_k(M)t^k
\end{equation}
where $b_k(M)$ is the $k$-th Betti number of $M$.
\end{definition}

\begin{definition}
 Let $f:M\rightarrow \mathbb{R}$ be a Morse function. Then, the Morse polynomial of $f$ is defined to be
\begin{equation}
P_f(t)=\sum_{k=0}^n \mu_kt^k
\end{equation}
where $\mu_k$ is the number of critical points of $f$ of index $k$.
\end{definition}

\begin{theorem}{(The Morse Inequality)}
Let $f:M\rightarrow\mathbb{R}$ be a Morse function on a smooth manifold $M$. Then, there exists a polynomial $R(t)$ with non-negative integer coefficients such that
\begin{equation*}
P_f(t)=P_M(t)+(1+t)R(t).
\end{equation*}
\end{theorem}
One may refer to \cite{Nicolaescu2011} for proof.

A Morse function $f:M\rightarrow\mathbb{R}$ is called a perfect Morse function if $P_f(t)=P_M(t)$ \cite{Nicolaescu2011}.
 
Now, we show that the function $f_C$ on $SO(n)$ defined in the previous section is also a perfect Morse function. 
\begin{theorem}
The function 
$$
f_C: SO(n) \rightarrow \mathbb{R}, \quad f_C (A):=\langle C,A \rangle
$$
is a perfect Morse function, where $C \in S_n (\mathbb{R} )$.
\end{theorem}

\begin{proof}
First we show that the Morse polynomial is,
\begin{equation} \label{Pf_C}
P_{f_C}(t)=(1+t)(1+t^2)\cdots(1+t^{n-1}).
\end{equation}
We use induction method. For making it easier, we label the function $f_C$ with $n$ as ${f_C}_n:SO(n)\rightarrow\mathbb{R}$. \\
Trivially, for $n=1$, $P_{{f_C}_1}(t)=1$ and for $n=2$, $P_{{f_C}_2}(t)=1+t$. Assume that, $P_{{f_C}_n}(t)=(1+t)(1+t^2)+\cdots+(1+t^{n-1})$. Then, we need to show that $P_{{f_C}_{n+1}}(t)$ satisfies the form \eqref{Pf_C}. \\
We may consider that $SO(n+1)$ gets all the critical points from $SO(n)$ with extra bottom entry ($(n+1)$- th diagonal entry), which is either +1 or -1. Say the set of all these points are $C_{n+1}^+$ and $C_{n+1}^-$ respectively.\\
 Let $A\in C_{n+1}^-$. Then we have $\tilde{A}\in O(n)$ such that, $A$ is the matrix $\tilde{A}$ with extra bottom entry -1. Then, by the definition of index, we obtain
\begin{equation*}
ind(A)=ind(\tilde{A}).
\end{equation*}
Thus, for the elements of $ C_{n+1}^-$ the equation \eqref{Pf_C} holds.
Let $A\in C_{n+1}^+$. Then we have $\tilde{A}\in SO(n)$ such that, $A$ is the matrix with $\tilde{A}$  with the bottom entry +1. Thus, by the definition of index, we obtain
\begin{equation*}
ind(A)=ind(\tilde{A})+n.
\end{equation*}

So, by the definition of Morse polynomial, we gain

\begin{equation}
P_{{f_C}_{n+1}}(t)=P_{{f_C}_n}(t)(1+t^n)=(1+t)(1+t^2)\cdots(1+t^{n-1})(1+t^n)
\end{equation}.

Now, we find out the Poincar\'e polynomial of $SO(n)$. The graded abelian group $H_*(SO(n),\mathbb{Z}_2)$ is isomorphic to the graded group coming from the exterior algebra
$$\wedge_{\mathbb{Z}_2}[e_1,e_2,...,e_{n-1}]
$$ \cite{Hatcher2002}.
Let say $A(n)=\wedge_{\mathbb{Z}_2}[e_1,e_2,...,e_{n-1}]$ where the degree of $e_i$, $|e_i|=i$. Then, we obtain
\begin{equation*}
|e_{i_1}\wedge e_{i_2}\wedge ...\wedge e_{i_k}|=\sum_{j=1}^k|e_{i_j}|=\sum_{j=1}^ki_j.
\end{equation*}
 If we define $a(n)_k=dim{\mathbb{Z}_2}(A(n)_k)$, then by the result in \cite{Hatcher2002}, $ a(n)_k$ is nothing but the $k$- th Betti number of $SO(n)$. Hence, the polynomial
 \begin{equation*}
 P(A(n))=\sum_{i=0}^\infty a(n)_it^i
 \end{equation*}
 is the Poincar\'e polynomial of $SO(n)$.

Now, our claim is that The Poincar\'e polynomial of $SO(n)$ is
\begin{equation*}
P(A(n))=(1+t)(1+t^2)\cdots(1+t^{n-1}).
\end{equation*}

 Let $B(A(n))$ be the basis of $A(n)$. For instance, $B(A(1))=trivial$, $B(A(2))=\lbrace 1,e_1 \rbrace $, $B(A(3))=\lbrace 1, e_1, e_2, e_1\wedge e_2 \rbrace$ etc.

 In this sense, we obtain
 \begin{equation*}
 B(A(n+1))=(B(A(n))\wedge e_n)\sqcup B(A(n)).
 \end{equation*}

We use induction method. Indeed, here we have very similar arguments with the previous claim. The variable $e_n$ has the same role with " the extra bottom entry $\pm1$". Then, we have the polynomial $P(A(n))=\sum_{b\in B(A(n))}a(n)_bt^{|b|}$. Trivially, $P(A(1))=1$ and $P(A(2))=1+t$. By the induction hypothesis, assume that
 \begin{equation*}
 P(A(n))=\sum_{b\in B(A(n))}a(n)_bt^{|b|}=(1+t)(1+t^2)\cdots(1+t^{n-1}).
 \end{equation*}
For the polynomial $P(A(n+1))$, pick an element $b\in B(A(n+1))$. Then, $b$ is in either $B(A(n))$ or $B(A(n))\wedge e_n$. For $b\in B(A(n+1))$, trivially, $P(A(n+1))$ has the desired form. If $b\in B(A(n))\wedge e_n$, then by the definition of degree, there is $\tilde{b}\in B(A(n))$ such that $|b|=|\tilde{b}|+n$. Thus, by the definition of $P(A(n))$, we obtain
\begin{equation}
P(A(n+1))=(1+t)(1+t^2)\cdots(1+t^n)
\end{equation}
which completes the proof.
 \end{proof}

Thereby, we have shown that, for the given Morse function $f_C:SO(n)\rightarrow\mathbb{R}$, $P_M(t)=P_{f_C}(t)$, meaning that $f_C$ is a perfect Morse function.

\bibliography{refs}

\begin{thebibliography}{1}

\bibitem{Gauld1982}
D.~B. Gauld.
\newblock {\em Differential topology: An Introduction}.
\newblock Marcel Dekker, New York, 1982.

\bibitem{Hatcher2002}
A.~Hatcher.
\newblock {\em Algebraic Topology}.
\newblock Cambridge University Press, New York, 3rd edition, 2002.

\bibitem{Lee2012}
J.~M. Lee.
\newblock {\em Introduction to Smooth Manifolds}.
\newblock Springer, New York, 2rd edition, 2012.

\bibitem{Matsumoto2002}
Y.~Matsumoto.
\newblock {\em An Introduction to Morse Theory}.
\newblock American Mathematical Society, 2002.

\bibitem{Milnor1963}
J.~W. Milnor.
\newblock {\em Morse Theory}.
\newblock Princeton University Press, 1963.

\bibitem{Morse1928}
M.~Morse.
\newblock The foundations of a theory of the calculus of variations in the
  large in m-space.
\newblock {\em Trans. Amer. Math. Soc.}, 30:213--274, 1928.

\bibitem{Nicolaescu2011}
L.~I. Nicolaescu.
\newblock {\em An Invitation to Morse Theory}.
\newblock Springer, New York, 2rd edition, 2011.

\end{thebibliography}
\bibliographystyle{plain}
\end{document}